\documentclass[11pt,leqno]{article}
\usepackage{amsmath}
\usepackage{verbatim}
\usepackage{amsfonts}
\usepackage{amssymb}
\usepackage{mathrsfs}
\usepackage{amsthm}
\usepackage{mathtools}
\usepackage{color} 

\usepackage{amsxtra}
\usepackage[greek,english]{babel}

\newtheorem{theorem}{Theorem}[section]
\newtheorem{corollary}{Corollary}[section]

\newtheorem{remark}{Remark}[section]
\newtheorem{definition}{Definition}[section]

\usepackage{chngcntr}
\counterwithin*{equation}{section}

\begin{document}

\title{
Approximations with Non-Symmetric Green's Kernels and their Application to Fractional Differential Equations
}
\date{}
\author
{ Nick Fisher
\thanks{Department of Mathematics and Statistics,
Portland State University, Portland, OR
E-mail:{\it nicholfi@pdx.edu}
}}
\maketitle

\begin{abstract}
Several kernel-based methods for the numerical solution of fractional differential equations have been developed in the recent past; however, these techniques exclusively relied on the use of radial basis function approximations. In the present work, we consider the non-symmetric Green's kernel perspective on fractional order spline interpolation and its application to a kernel Galerkin method for the numerical solution of certain fractional order differential equation. Unfortunately, the reliance on a non-symmetric kernel requires that our theoretical analysis of the kernel interpolants must take place outside the familiar setting of reproducing kernel Hilbert spaces. Nevertheless, we are able to prove that the proposed kernel interpolants obtain optimal order convergence rates in a reproducing kernel Banach space.
\end{abstract}

{\bf Key words.} 
fractional derivatives, Green's kernels, interpolation, approximation order, reproducing kernel Banach spaces

{\bf AMS subject classification.} 41A15, 41A65, 65D05


\section{Introduction} 
The present work is concerned with the fractional Green's kernel approach to scattered data interpolation and its application to the numerical solution of fractional differential equations. That is, given data $\{( x_i, y_i )\}_{i = 1}^N$ with $ x_i \in \mathbb{R}$, $y_i \in \mathbb{R}$, we would like to find a continuous function $s$ such that $s(x_i) = y_i$, $i = 1, \dots, N$. In particular, we would like to express the interpolant as a linear combination of  basis functions $\{B_j\}_{j = 1}^N$, so that
\[s( x)  = \sum_{j = 1}^Nc_jB_j( x), \quad  x \in \mathbb{R}.\]
We will choose a basis of the form
\[\mathcal B = \{K(\cdot, x_1),\dots,K(\cdot,x_N) \}\]
so that our interpolant can be expressed as
\begin{equation}\label{eq: general interpolant}
    s(x)  = \sum_{j = 1}^Nc_jK(x,  x_j) = \bold k( x)^T\bold c, \quad  x \in \mathbb{R}.
\end{equation}
Then, the coefficients $c_j$ are obtained by solving the linear system
\[\text{K}\bold c = \bold y\]
with $(\text{K})_{i,j} = K( x_i,  x_j)$, $i,j = 1, \dots, N$. In other words $\bold c = \text{K}^{-1}\bold y$ and the kernel interpolant can be evaluated as
\[s(x) = \bold k( x)^T\text{K}^{-1}\bold y,\]
where $\bold k( x) = \left( K( x,x_1),\dots,K( x, x_N) \right)$. The existing literature on kernel methods offers a wide assortment of choices for the kernel $K( x, z)$ \cite{Fass}. This allows practitioners great flexibility in dealing with large data sets in a mesh-free setting while simultaneously allowing for great specificity in the analytical properties of the resulting interpolant.
For instance, one may choose from periodic kernels \cite{Wahba}, discontinuous kernels \cite{DeMarchi}, or divergence-free kernels \cite{Fus} among many other possibilities. 
In recent years, many approaches to fractional order approximation have been proposed, such as fractional B-splines \cite{Unser} and spectral collocation methods \cite{ZayKar}. Additionally, kernel methods for fractional differential equations were studied in \cite{Lischke, Moh, PiretHanert}, but this research was largely based on applications of radial basis functions. With all this in mind, we introduce the Green's kernel perspective on fractional order spline interpolation by building on previous research regarding integer order Brownian bridge kernels (see Chapter 7 in \cite{Fass}, for instance) in reproducing kernel Hilbert spaces (RKHS's). Owing to the fact that the proposed kernels are not symmetric (see Section 2.2 for details), much of the theoretical machinery regarding reproducing kernel Hilbert spaces is not applicable. Consequently, we give some cursory details on the broader theory of reproducing kernel Banach spaces (RKBS's) \cite{XuYe} and hope to place the proposed methods on solid theoretical ground. We aim to show that the new fractional Green's kernel techniques provide unique perspective for understanding the connection between RKHS's and RKBS's as the fractional kernels contain as special cases the well understood integer order Brownian bridge kernels. Finally, to establish theoretical bounds on the errors in fractional order Green's kernel interpolation, we investigate the current literature on fractional order sampling inequalities \cite{Arcangeli,Corrigan}. The numerical results are encouraging.

As an application of the proposed interpolation methods, we consider the kernel Galerkin discretization of certain fractional differential equations. Of particular interest is the fractional diffusion equation of the form 
\begin{equation}\label{eq: frac diff 1}
    \frac{\partial u(x,t)}{\partial t} - d(x) {}_0D_x^\alpha u(x,t)  = q(x,t), \quad x \in (0,1),\, t>0,
\end{equation} 
\begin{equation} u(x,0) = g(x), \quad x \in (0,1),\end{equation} 
\begin{equation}\label{eq: frac diff 3}u(0,t) = u(1,t) = 0, \quad t \geq 0,\end{equation} 
where ${}_0D_x^\alpha$ is the left Riemann-Liouville operator of order $\alpha$ based at $0$ (to be defined shortly). Equations \eqref{eq: frac diff 1}-\eqref{eq: frac diff 3} have been solved by a variety numerical of methods, including finite difference \cite{Fisher,Tadjeran,Meerschaert,Liu1,Liu2,Lynch,TMS}, finite element \cite{Ervin,Roop}, and spectral/pseudospectral methods \cite{PiretHanert, ZayKar}. Recently, the kernel Galerkin approach was applied to the solution of the fractional diffusion equation in \cite{Moh}. A key contrast with the present work is that the numerical methods considered in \cite{Moh} involved Green's functions of the \emph{self adjoint} fractional Laplacian. Whereas, in the proposed method, the discretization of equation \eqref{eq: frac diff 1} involves the Green's function of the \emph{non}-self adjoint Riemann-Liouville derivative. It is because of this distinction between the self adjoint derivative operators in \cite{Moh} and non-self adjoint operators in \eqref{eq: frac diff 1} that necessitates the move from the RKHS setting to the RKBS setting, since we do not expect the Green's function of the non-self adjoint Riemann-Liouville operator to be symmetric.

The remainder of the paper is organized as follows. Sections 2 and 3 provide relevant background information of fractional calculus and reproducing kernel Banach spaces, respectively. In Section 4, the approximation orders of the proposed interpolations method are derived. The kernel Galerkin methods for fractional differential equations are described in Section 5 while numerical results are listed in Section 6. Finally, Section 7 contains some concluding remarks. 

\section{Fractional Calculus}

\subsection{Some function spaces}
We begin with a review of some function spaces which are frequently used in the analysis of problems involving fractional order differential operators. Throughout we assume $\Omega = (a,b) \subset \mathbb{R}$ and $\partial \Omega = \{a,b\}$ so that $\overline \Omega  = \Omega \cup \partial \Omega = [a,b]$.\\

The Lebesgue space, $L^p(\Omega)$, for $1\leq p \leq \infty$, consists of the set of functions $v: \Omega \rightarrow \mathbb{R}$ with a finite norm $\|v\|_{L^p(\Omega)}$ defined by
\[ \|v\|_{L^p(\Omega)} = \left\{ \begin{array}{cc}
    \left(\int_\Omega|v|^pdx\right)^{1/p}, & \text{if } 1\leq p < \infty \\
    \text{ ess } \sup_\Omega |p|, & \text{if } p = \infty \
\end{array}\right. .\]
Additionally, for two functions $u$ and $v$ defined over $\Omega$, we denote
\[\langle u , v \rangle = \int_\Omega uv \, dx\]
when the domain of integration is understood from context.
Let $s\in \left (\tfrac{1}{2},1 \right )$, then the fractional Sobolev space $W_p^s(\Omega)$ is defined by
\[W_p^s(\Omega) = \left \{v \in L^p(\Omega) :  |v|_{W_p^s(\Omega)} < \infty \right\},\]
where the Sobolev-Slobodekij seminorm $|\cdot|_{W_p^s (\Omega)}$ is defined by
\[|v|_{W_p^s (\Omega)} = \left(\int_{\Omega}\int_{\Omega} \frac{(v(x) - v(y))^p}{|x-y|^{d+ps}}dx \, dy\right)^{1/p}.\]
Note that the space $W_p^s(\Omega)$ is a Banach space when it is equipped with the norm
\[\|v\|_{W_p^s(\Omega)} = \left( \| v\|^p_{L^p(\Omega)}  +|v|^p_{W_p^s} (\Omega) \right)^{1/p}.\]
Next, we define the Hilbert spaces
\[H^s(\Omega) = W_2^s(\Omega) \text{ and } H_0^s(\Omega) = \left\{v \in H^s(\Omega) : \left.v\right|_{\partial \Omega} = 0 \right\}.  \]
The space of absolutely continuous functions $AC^n(\overline \Omega)$ is defined as follows. A function $v: \Omega \rightarrow \mathbb{R}$ is called absolutely continuous on $\overline \Omega$ if for any $\epsilon> 0$, there exists a $\delta > 0$ such that for any finite set of pairwise disjoint intervals $[a_k,b_k] \subset \overline \Omega, $ $k = 1, \dots, n$, such that $\sum_{k= 1}^n (b_k - a_k) < \delta$, there holds the inequality
\[\sum_{k = 1}^n |v(b_k) - v(a_k))| < \epsilon . \]
The space of these functions is denoted by $AC(\overline \Omega)$. Furthermore, $AC^n(\overline \Omega)$, $n \in \mathbb{N}$, denotes the space of functions $v$ with continuous derivatives up to order $n-1$ in $\overline \Omega$ and the $(n-1)^{st}$ derivative $v^{(n-1)}\in AC(\overline \Omega)$.

\subsection{Fractional derivatives}
With the above definitions in mind, we mention some basic definitions of fractional order derivatives and their properties \cite{Jin, Pod, Samko}. In particular, there are many possible definitions of the fractional derivative. We mention a few popular choices beginning with the Riemann-Liouville derivative.
\begin{definition}
    For $f \in L^1(\Omega)$ and $n-1 < \alpha \leq n$, $n \in \mathbb{N}$, its left-sided Riemann-Liouville fractional derivative of order $\alpha$ (based at $x = a$), denoted by ${}_a D_x^\alpha f $, is defined by 
    \[{}_a D_x^\alpha f(x) = \frac{1}{\Gamma(n-\alpha)}\frac{d^n}{dx^n}\int_a^x(x-\tau)^{n -\alpha-1}f(\tau)d\tau,  \]
    (provided the integral exists) and its right-sided Riemann-Liouville fractional derivative of order $\alpha$ (based at $x = b$), denoted by ${}_x D_b^\alpha f $, is defined by 
    \[{}_x D_b^\alpha f(x) = \frac{(-1)^n}{\Gamma(n-\alpha)}\frac{d^n}{dx^n}\int_x^b(\tau-x)^{n -\alpha-1}f(\tau)d\tau,  \]
    (provided the integral exists).
\end{definition}
Next, the Caputo derivative.
\begin{definition}
    For $f \in L^1(\Omega)$ and $n-1 < \alpha \leq n$, $n \in \mathbb{N}$, its left-sided Caputo fractional derivative of order $\alpha$ (based at $x = a$), denoted by ${}^C_a D_x^\alpha f $, is defined by 
    \[{}^C_a D_x^\alpha f(t) = \frac{1}{\Gamma(n-\alpha)}\int_a^x(x-\tau)^{n -\alpha-1}f^{(n)}(\tau)d\tau,  \]
    (provided the integral exists) and its right-sided Caputo fractional derivative of order $\alpha$ (based at $x = b$), denoted by ${}^C_x D_b^\alpha f $, is defined by 
    \[{}^C_x D_b^\alpha f(x) = \frac{(-1)^n}{\Gamma(n-\alpha)}\int_x^b(\tau-x)^{n -\alpha-1}f^{(n)}(\tau)d\tau,  \]
    (provided the integral exists).
\end{definition}

\subsection{Fractional Green's kernels}
We begin this section with a review of the traditional approach to Green's kernels.
\begin{definition}
Given a linear (ordinary or partial) differential operator $\mathcal{L}$ on the domain 
$\Omega \subseteq \mathbb{R}^d$, the Green's kernel $G$ of $\mathcal{L}$ is defined as 
the solution of 
\[\mathcal{L}G(\bold x,\bold z) = \delta(\bold x - \bold z), \quad \bold z \in \Omega \text{ fixed}.\]
Here $\delta(\bold x -\bold z) $ is the \emph{Dirac delta functional} evaluated at $\bold x - \bold z$, i.e., $\delta(\bold x - \bold z) = 0$ for $\bold x \neq \bold z$ and $\int_\Omega \delta (\bold x) d \bold x = 1$.
In what follows, we often use the compact notation $\delta_{\bold z}(\bold x) \coloneqq \delta(\bold x -\bold z)$.
\end{definition}
In particular, if $G$ and $\mathcal{L}$, defined above, satisfy $\mathcal{L}u = f$, then
\[u(\bold x) = \int_\Omega G(\bold x, \bold z)f(\bold z) d\bold z.\]

As an example, consider the Brownian bridge kernel which is the Green's kernel of the two-point boundary value problem 
\[
\left\{
\begin{array}{rcl}
     -u''(x) & = & f(x), \quad \text{in }\Omega = (0,1), \\
       u(0) & = & u(1) = 0, \\
\end{array} 
\right.
\]
given by
\[G(x,z)  = \left\{
\begin{array}{ll}
     (1-z)x & 0\leq x \leq  z \leq 1 \\
      -z(x-1) & 0 \leq z \leq x \leq 1 \\
\end{array} 
\right. \]
or
\begin{equation}\label{eq: bb kernel}
G(x,z) = \min(x,z) - xz.
\end{equation}
This choice of kernel and the corresponding approximation methods have been well studied \cite{Cavoretto,Fass}. In particular, it can be shown that the Green's kernel \eqref{eq: bb kernel} is the reproducing kernel of the Hilbert space $H_0^1([0,1])$. Hence, generalizations of two-point boundary value problems of this form will make an excellent starting place for the study of fractional Green's kernels and their approximation properties. A plot of the kernel evaluated at equally spaced points on $[0,1]$ is given in Figure \ref{fig:bb}.

\begin{figure}[ht]
\centering
\includegraphics[width=.5\textwidth]{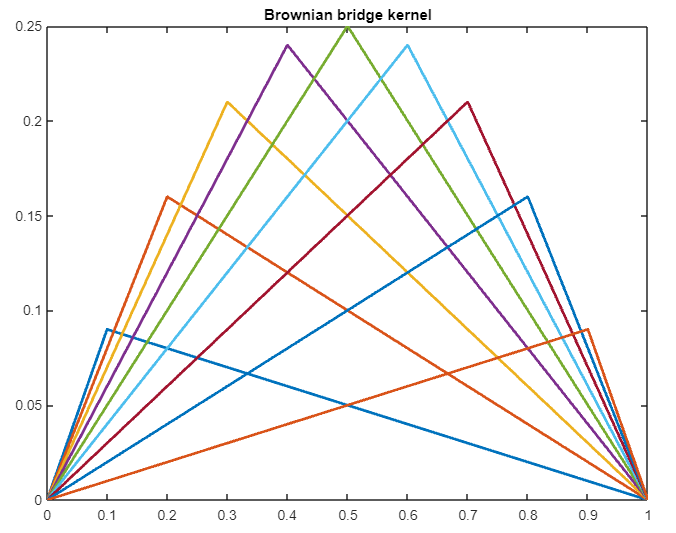}
\caption{The Brownian bridge kernel evaluated at equally spaced points on $[0,1]$.}
\label{fig:bb}
\end{figure}

What follows are two examples of fractional Green's kernels. Each kernel can be thought of as a fractional generalization of the Brownian bridge kernel. 
First, the Riemann-Liouville Green's function. 
\begin{theorem}\label{thm:RL Green's function}
Let $f:\Omega = (0,1)\times\mathbb{R} \to \mathbb{R}$ be such that $f(x,u(x)) \in L^1(\Omega)$ for 
$u\in L^1(\Omega)$. Then a function $u\in L^1(\Omega)$ with ${}_0 D_t^{-(\alpha-2)} u \in AC^2(\Omega)$
solves
\[
\left\{
\begin{array}{rcl}
     -{}_0 D_x^{\alpha} u(x) & = & f(x,u(x)), \quad \text{in }\Omega \\
      \left({}_0 D_x^{\alpha-2} u\right)(0) & = & u(1) = 0 \\
\end{array} 
\right.
\]
if and only if $u \in L^1(\Omega) $ satisfies $u(x) = \int_0^1G(x,z)f(z,u(z))dz$
where the Green's function $G(x,z)$ is given by
\[
G(x,z) = \frac{1}{\Gamma(\alpha)}\left\{
\begin{array}{rl}
     (x(1-z))^{\alpha - 1} - (x-z)^{\alpha - 1},&   0\leq z\leq x \leq 1  \\
      x^{\alpha - 1}(1-z)^{\alpha - 1}, &  0\leq x\leq z \leq 1 \\
\end{array} 
\right. .
\]
\end{theorem}  
\begin{proof}
    This is Theorem 5.1 in \cite{Jin}.
\end{proof}
\begin{remark}
    The Neumann-type boundary condition $\left({}_0 D_x^{\alpha-2} u\right)(0) = 0$ of Theorem \ref{thm:RL Green's function} may evidently be replaced by the Dirichlet boundary condition $ u(0) = 0$ \cite{Jin}.
\end{remark}

\begin{figure}
\centering
\includegraphics[width=.4\textwidth]{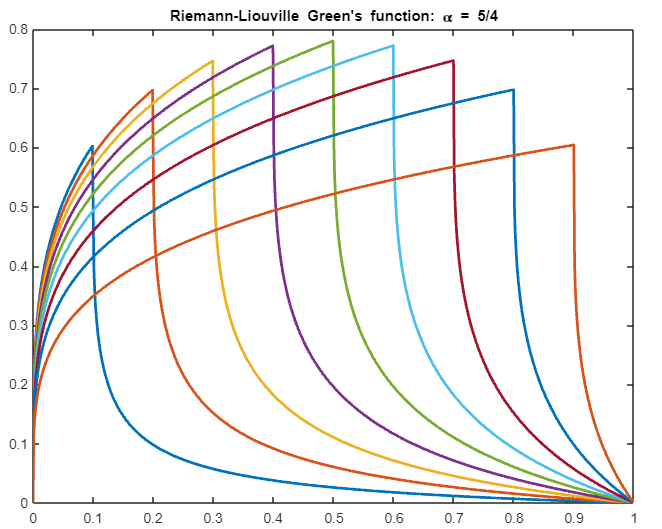}
\includegraphics[width=.4\textwidth]{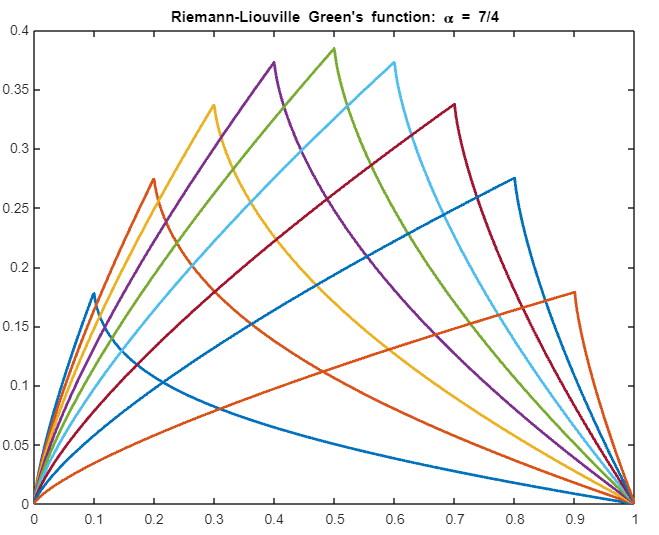}\\
\includegraphics[width=.4\textwidth]{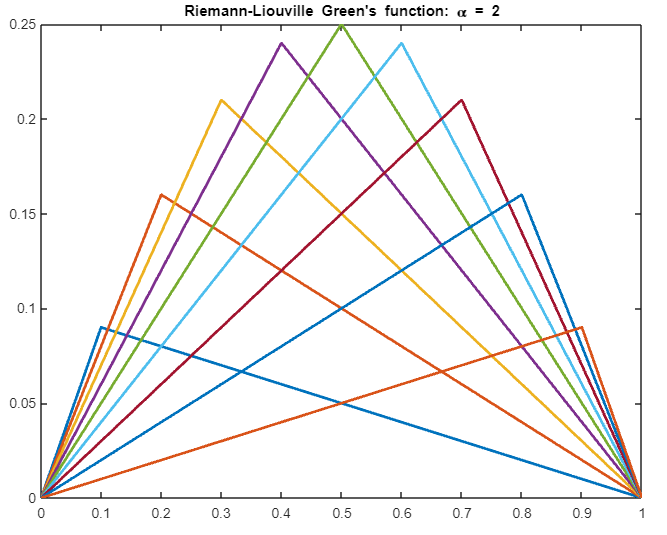}
\includegraphics[width=.4\textwidth]{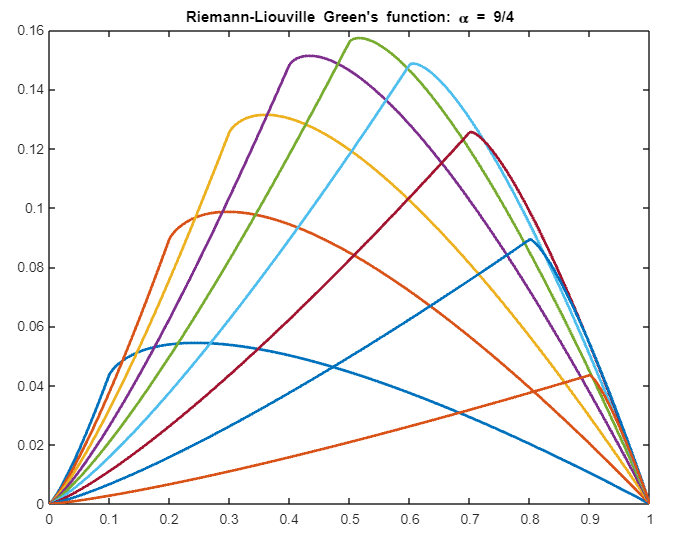}
\caption{Riemann-Liouville Green's kernels evaluated at equally spaced points on $[0,1]$.}
\end{figure}

Next, for the sake of completeness, we include the Caputo Green's function. However, for ease of presentation, in the theoretical analysis and numerical experiments (Sections \ref{sec: approximation orders} and \ref{sec: numerics}, respectively) we focus solely on the Riemann-Liouville case.

\begin{theorem}
Let $f:\Omega = (0,1)\times\mathbb{R} \to \mathbb{R}$ be continuous. Then a function $u\in C^1(\overline \Omega)$ 
solves
\[
\left\{
\begin{array}{rcl}
     -{}_0^C D_x^{\alpha} u(x) & = & f(x,u(x)), \quad \text{ in } \Omega \\
       u(0) & = & u(1) = 0 \\
\end{array} 
\right.
\]
if and only if $u \in C^1(\overline \Omega) $ satisfies $u(x) = \int_0^1G(x,z)f(z,u(z))dz$
where the Green's function $G(x,z)$ is given by
\[
G(x,z) = \frac{1}{\Gamma(\alpha)}\left\{
\begin{array}{rl}
     x(1-z)^{\alpha - 1} - (x-z)^{\alpha - 1},&   0\leq z\leq x \leq 1  \\
      x(1-z)^{\alpha - 1}, &  0\leq x\leq z \leq 1 \\
\end{array} 
\right. .
\]
\end{theorem}    
\begin{proof}
    This is Corollary 5.1 in \cite{Jin}.
\end{proof}

\begin{figure}
\centering
\includegraphics[width=.4\textwidth]{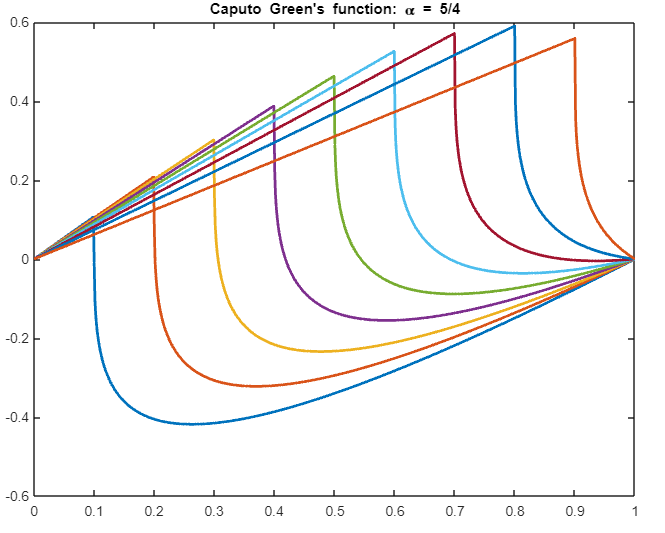}
\includegraphics[width=.4\textwidth]{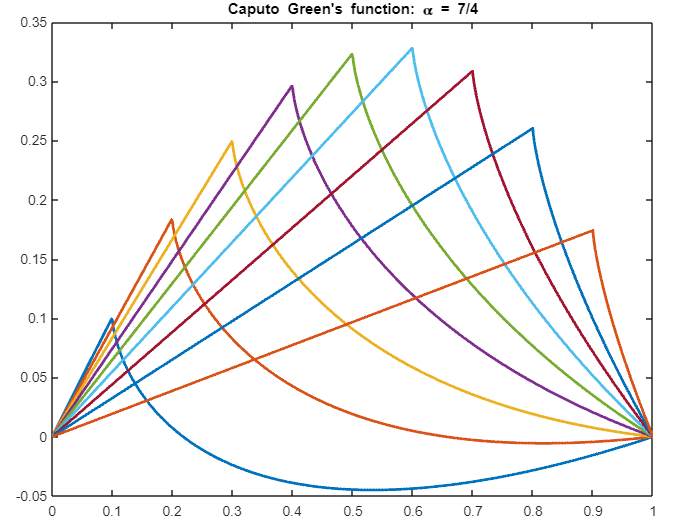}\\
\includegraphics[width=.4\textwidth]{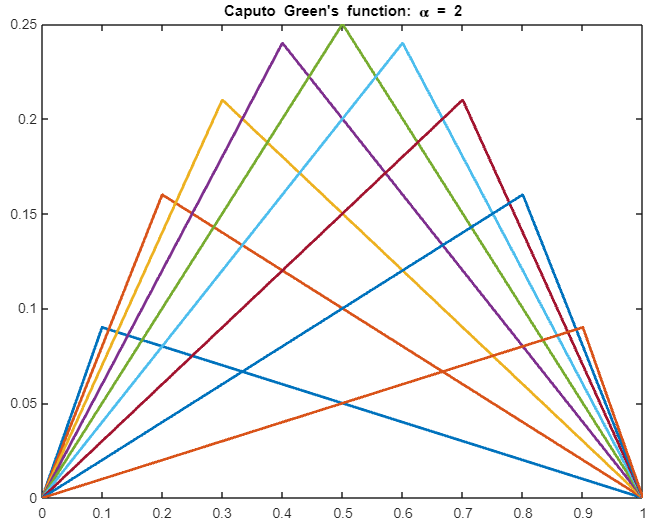}
\includegraphics[width=.4\textwidth]{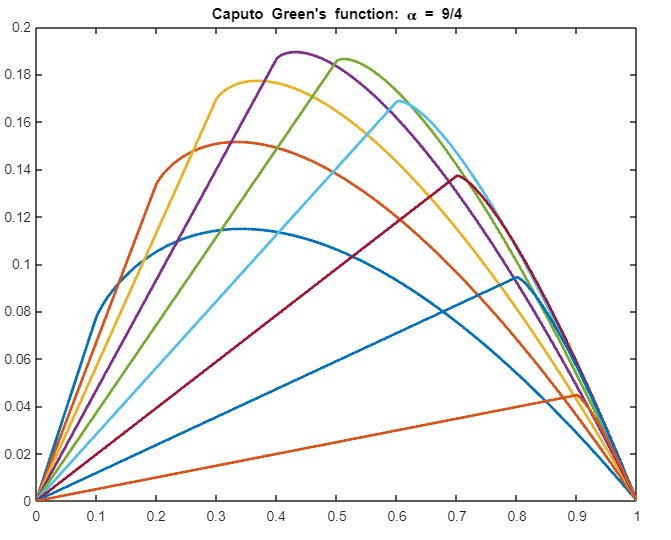}
\caption{Caputo Green's kernels evaluated at equally spaced points on $[0,1]$.}
\end{figure}

\subsection{Variational formulation}
The variational formulation \cite{Ervin,Jin} of the fractional boundary value problem  
\begin{equation} \label{BVP}
\left\{
\begin{array}{rcl}
     -{}_0 D_x^{\alpha} u(x) & = & f(x,u(x)), \quad \text{in }\Omega = (0,1) \\
       u(0) & = & u(1) = 0 \\
\end{array} 
\right.
\end{equation}
is to find the $u \in H^{\alpha/2}_0(\Omega)$ satisfying 
\[A(u,v) = \langle f,v\rangle, \quad \forall v \in H^{\alpha/2}_0(\Omega),\]
where the bilinear form on the left-hand side is given by
\begin{equation}\label{eq: bilinear form}
A(u,v) = -\langle{}_0 D_x^{\alpha} u,v\rangle = -\langle{}_0 D_x^{\alpha/2} u,{}_x D_1^{\alpha/2} v\rangle.
\end{equation}
The bilinear form, \eqref{eq: bilinear form}, will play a central role in the remainder of the paper.

The next theorem implies the stability of the variational formulation.

\begin{theorem}\label{thm: variational stability}
    For $\Omega = (0,1)$ and $\gamma \in (1/2,1)$, there exists a $c > 0$ depending on $\gamma$ which satisfies
    \[c\|u \|^2_{H_0^\gamma(\Omega)} \leq -\langle{}_0 D_x^{\gamma} u,{}_x D_1^{\gamma}u\rangle = A(u,u), \quad \forall u \in H_0^\gamma(\Omega).\]
\end{theorem}
\begin{proof}
    This is Lemma 5.3 in \cite{Jin}.
\end{proof}
The following corollary will be helpful in proving the convergence rates of the kernel interpolants.
\begin{corollary}\label{the corollary}
    For $\Omega = (0,1)$ and $\gamma \in (1/2,1)$, there exists a $C > 0$ depending on $\gamma$ which satisfies
    \[|u |^2_{H_0^\gamma(\Omega)} \leq C|\langle{}_0 D_x^{\gamma} u,{}_x D_1^{\gamma}u\rangle|, \quad \forall u \in H_0^\gamma(\Omega).\]
\end{corollary}
\begin{proof}
    This follows directly from Theorem \ref{thm: variational stability} and the definition of $\|\cdot\|_{H_0^\gamma(\Omega)}$.
\end{proof}

\section{Reproducing Kernel Banach Spaces}
Except for some special cases, we note that the Green's kernels proposed above are not symmetric ($\alpha = 2$ is the standard Brownian bridge kernel). This runs counter to the general theory of reproducing kernel Hilbert spaces. In particular, for our purposes, a \emph{kernel} $K$ is a real-valued function of two variables:
\[K: \Omega \times \Omega \to \mathbb{R}, \quad K: ( x,  z) \mapsto K( x, z)\]
We would like our kernel, $K$, to be (strictly) \emph{positive definite}. This means that the (symmetric) kernel matrix $\text{K}$ with entries 
\[(\text{K})_{i,j} = K( x_i, x_j), \quad  i,j = 1, \dots, N\]
is positive definite for any $N \in \mathbb{N}$ and for any set $\mathcal{X} = \{ x_1, \dots,  x_N \} \subset \Omega$ of distinct points. This will ensure that the matrix $\text{K}$ is invertible. Special cases aside, if we take the Green's kernel, $G(x,z)$, of Theorem 2.1 or 2.2 it is clear that the corresponding interpolation matrix with entries
\[(\text{G})_{i,j} = G( x_i, x_j), \quad  i,j = 1, \dots, N\]
is not necessarily positive definite since the matrix $\text{G} $  will not be symmetric. This leads us to the following:

\subsection{Reproducing kernels for RKBS's}

In several recent papers, reproducing kernel Hilbert spaces have been generalized to reproducing kernel Banach spaces \cite{XuYe, Lin, ZXZ}. In an early paper, the RKBS was constructed with the inner product replaced by a semi-inner product (which is not necessarily symmetric anymore)\cite{ZXZ}.  More recently, RKBSs have been described by means of a bilinear form on a pair of Banach spaces \cite{XuYe, Lin}. The following is the definition  of reproducing kernels for RKBS given in \cite{Lin}:
\begin{definition} \label{def: reproducing kernel}
Let $\mathcal{B}_1$ be an RKBS on a set $\Omega_1$. If there exists a Banach space $\mathcal{B}_2$ of functions on another set $\Omega_2$, a continuous bilinear form $\langle \cdot, \cdot\rangle_{\mathcal{B}_1 \times \mathcal{B}_2}$, and a function $K$ on $\Omega_1 \times \Omega_2$ such that $K(x,\cdot) \in\mathcal{B}_2  $ for all  $x \in \Omega_1$
 and $f(x)=\langle f,K(x,\cdot)\rangle_{\mathcal{B}_1 \times \mathcal{B}_2}$
 for all $x \in \Omega_1$ and all $f \in \mathcal{B}_1$, then we call $K$ a reproducing kernel for $\mathcal{B}_1$.

If, in addition, $\mathcal{B}_2$ is also an RKBS on $\Omega_2$ and it holds $K(\cdot,y) \in \mathcal{B}_2$ for all $y \in \Omega_2$ and $g(y)=\langle K(\cdot,y),g\rangle_{\mathcal{B}_1 \times \mathcal{B}_2}$ for all $y \in \Omega_2$ and all $g \in \mathcal{B}_2$, then we call $\mathcal{B}_2$ an adjoint RKBS of $\mathcal{B}_1$
 and call $\mathcal{B}_1$ and $\mathcal{B}_2$ a pair of RKBSs.

In this case, $K'(x,y):=K(y,x)$, $x \in \Omega_2$, $y \in \Omega_1$, is a reproducing kernel for $\mathcal{B}_2$.
\end{definition}
In short, we should be able to show that the fractional Green's kernels are reproducing kernels for some RKBS and that, in some special cases (e.g., $\alpha = 2$), these Banach spaces are, in fact, the Hilbert spaces reproduced by the integer order Brownian bridge kernels. 

\subsection{Constructing the RKBS's via feature maps}

What follows is the generic construction in \cite{Lin} of an RKBS via a pair of feature maps.

Let $\mathcal{W}_1$, $\mathcal{W}_2$ be two Banach spaces, and $\langle \cdot, \cdot \rangle_{\mathcal{W}_1\times \mathcal{W}_2}$ be a continuous bilinear form on $\mathcal{W}_1\times \mathcal{W}_2$. Suppose there exists two nonempty sets $\Omega_1$ and $\Omega_2$, and mappings
    \[\Phi_1: \Omega_1 \rightarrow \mathcal{W}_1, \quad
    \Phi_2: \Omega_2 \rightarrow \mathcal{W}_2\]
such that with respect to the bilinear form 
\[\text{span } \Phi_1(\Omega_1) \text{ is dense in } \mathcal{W}_1, \quad \text{span } \Phi_2(\Omega_2) \text{ is dense in } \mathcal{W}_2.\]
We construct 
\[\mathcal{B}_1 \coloneqq \left\{ f_v(x) \coloneqq \langle \Phi_1(x), v \rangle_{\mathcal{W}_1\times \mathcal{W}_2} : v \in \mathcal{W}_2, x \in \Omega_1\right\}\]
with norm  
\[\|f_v\|_{\mathcal{B}_1} \coloneqq \|v\|_{\mathcal{W}_2}\]
and 
\[\mathcal{B}_2 \coloneqq \left\{ g_u(y) \coloneqq \langle u, \Phi_2(y) \rangle_{\mathcal{W}_1\times \mathcal{W}_2} : u \in \mathcal{W}_1, y \in \Omega_2\right\}\]
with norm  
\[\|g_u\|_{\mathcal{B}_2} \coloneqq \|u\|_{\mathcal{W}_1}.\]
Taking $\mathcal{B}_1$ and $\mathcal{B}_2$ defined above with the bilinear form on $\mathcal{B}_1 \times \mathcal{B}_2$ defined as
\[\langle f_v , g_u \rangle_{\mathcal{B}_1 \times \mathcal{B}_2} \coloneqq \langle u,v \rangle_{\mathcal{W}_1 \times \mathcal{W}_2}, \quad \forall f_v \in \mathcal{B}_1, g_u \in \mathcal{B}_2\]
Theorem 2.3 in \cite{Lin} guarantees that $\mathcal{B}_1$ is an RKBS on $\Omega_1$ with the adjoint RKBS $\mathcal{B}_2$ on $\Omega_2$.\\

Following the above construction, we take $\mathcal{W}_1 = \mathcal{W}_2 = H_0^{\alpha/2}(\Omega)$ with $\Omega_1 = \Omega_2 = \Omega = (0,1)$, then by Lemma 3.3 of \cite{Ervin} the bilinear form \eqref{eq: bilinear form} is a continuous bilinear form over $\mathcal{W}_1\times \mathcal{W}_2$, namely, we set

\[ \langle \cdot,\cdot\rangle_{\mathcal{W}_1 \times \mathcal{W}_2} = \langle{}_0 D_x^{\alpha/2} \cdot ,{}_x D_1^{\alpha/2} \cdot \rangle  .\] Next, we take the feature maps $\Phi_1$ and $\Phi_2$ to be the canonical feature maps defined by the Green's function $G(x,z)$ of Theorem \ref{thm:RL Green's function}, and $G'(x,z)$, the Green's function of the corresponding right Riemann-Liouville boundary value problem, respectively. That is, we define
\[\Phi_1(x) = G(\cdot,x), \text{ and } \Phi_2(y) = G'(\cdot,y).\]
For instance, the reproducing kernel, $K(x,y)$, of Definition \ref{def: reproducing kernel} can be computed as
\begin{multline}K(x,y) = \langle \Phi_1(x),\Phi_2(y)\rangle_{\mathcal{W}_1 \times \mathcal{W}_2} = \\ \langle{}_0 D_x^{\alpha/2} G(\cdot,x) ,{}_x D_1^{\alpha/2} G'(\cdot,y)\rangle  =  
\langle{}_0 D_x^{\alpha} G(\cdot,x) ,G'(\cdot,y)\rangle  =  \langle\delta_x (\cdot) ,G'(\cdot,y)\rangle = G'(x,y). \end{multline}

\section{Approximation Orders via Fractional Sampling Inequalities}\label{sec: approximation orders}

Given that approximation orders for (integer order) Green's kernel can be established through the use of sampling inequalities, we seek to do the same in the fractional order setting \cite{Arcangeli,Corrigan}. Fortunately, much of the theoretical work for RKBS's has been done. In particular, the optimality of the kernel interpolants has been proven (see Section 5 of \cite{ZXZ} or Section 2.6 of \cite{XuYe}, for example). Given this, the results of the numerical experiments of Section \ref{sec: interp probs} suggest that we should be able to prove that for some sufficiently smooth function $f$,
\[\|f-s\|_{L^2[0,1]} = \mathcal{O}(h^\alpha)\]
where $h$ is the fill distance defined by
\[h = \sup_{x \in \Omega} \min_{x_j \in \mathcal{X}}\|x-x_j \|\]
and $s$ is the kernel interpolant of the form \eqref{eq: general interpolant}.

\begin{theorem} \label{thm: convergence rates}
    Let $f\in \left\{C^1_0([0,1]): f' \in C_0([0,1))\right\}$ be interpolated by a left Riemann-Liouville Green's kernel, $G(x,z)$, with $1< \alpha < 2$ on the set $\mathcal{X} = \{x_1 \dots, x_N \}$. Then
    \[\|f-s\|_{L_2[0,1]} \leq C_3 h^{\alpha}\| {}_0 D_x^{\alpha}f\|_{L_2[0,1]}\]
    for some constant $C_3$ that depends on $\alpha$.
\end{theorem}
\begin{proof}
   Let $1 < \alpha < 2$, $p \in [1, \infty]$, $q \in [1, \infty)$,  and consider the fractional order sampling inequality of Theorem 3.2 in \cite{Arcangeli}:
    \[|u|_{W_q^m(\Omega)}\leq C_1\left(h^{\beta-m-\left(\frac{1}{p} - \frac{1}{q}\right)_{\hspace{-0.25em}+}}|u|_{W^\beta_p(\Omega)} + h^{-m}\| \bold u\|_{\infty}\right)\]
where $m\geq 0$ and $\beta>m+\frac{1}{p}$ define the \emph{strong} semi-norm $|\cdot|_{W^\beta_p(\Omega)}$ which together with the discrete norm of the values $ \bold u = (u( x_1)\dots u( x_N))^T$ of $u$ on the set $\mathcal{X} = \{ x_1, \dots,  x_N \} \subset \mathbb{R}$ defined by
\[\|\bold u\|_\infty = \max_{i = 1, \dots, N}|u(x_i)|\]with fill distance $h$ bound the \emph{weak} semi-norm $|\cdot|_{W_q^m(\Omega)}$.
Take $u$ to be the residual $f-s$ where $s$ interpolates $f$ on $\mathcal{X}$, then we have
\[|f-s|_{W_q^m(\Omega)}\leq C_1h^{\beta-m-d\left(\frac{1}{p} - \frac{1}{q}\right)_{\hspace{-0.25em}+}}|f-s|_{W^\beta_p(\Omega)},\]
since $u(x_i) = f(x_i) - s(x_i) = 0$, $i = 1, \dots, N$.
Additionally, if we assume $\Omega = [0,1]$, set $m = 0$, $\beta = \alpha/2$, let $p = q= 2$, then we have
\begin{equation}\label{eq: suboptimal bound 2}
\|f-s\|_{L_2([0,1])}\leq C_1 h^{\alpha/2}|f-s|_{H^{\alpha/2}([0,1])}
\end{equation}
or, 
\begin{equation}\label{eq: suboptimal bound 1}
\|f-s\|^2_{L_2([0,1])}\leq C_1^2 h^{\alpha}|f-s|^2_{H^{\alpha/2}([0,1])}
\end{equation}
provided $f,s \in H_0^{\alpha/2}([0,1])$. If we choose  $f\in \left\{C^1_0([0,1]): f' \in C_0([0,1))\right\}$ and $s$ defined by \eqref{eq: general interpolant} with the kernel given by the Green's function of Theorem \ref{thm:RL Green's function}, these conditions are satisfied. Notice that \eqref{eq: suboptimal bound 2} offers a suboptimal bound on the $L^2$ norm of the residual. We now show how the regularity assumptions on $f$ lead to an optimal error bound.

Consider, from Corollary \ref{the corollary}, that we have 
\[|f-s|^2_{H_0^{\alpha/2}([0,1])} \leq  C_2 \left|\left \langle {}_0 D_x^{\alpha/2} (f-s),{}_x D_1^{\alpha/2} (f-s)\right \rangle \right|, \quad \forall f-s \in H_0^{\alpha/2}([0,1]),\]
where $C_2$ depends on $\alpha$. However, for $f-s \in H_0^{\alpha/2}$ and $\mathcal{B}([0,1]) = H^{\alpha/2}_0([0,1]) \times H^{\alpha/2}_0([0,1])$
\[\left \langle {}_0 D_x^{\alpha/2} (f-s),{}_x D_1^{\alpha/2} (f-s)\right \rangle  = \left \langle  f-s, f-s\right \rangle_{\mathcal{B}([0,1])}, \]
and by orthogonality in RKBS's \cite{Fass2, XuYe, ZXZ}
\[\left \langle  f-s, f-s\right \rangle_{\mathcal{B}([0,1])} = \left \langle  f, f-s\right \rangle_{\mathcal{B}([0,1])}. \]
Next, by \eqref{eq: bilinear form}, the Cauchy-Schwarz inequality, and the smoothness of $f$, we have
\[| \left \langle f, f-s\right \rangle_{\mathcal{B}([0,1])}| \leq \|{}_0 D_x^{\alpha} f\|_{L^2([0,1])}\|f-s\|_{L^2([0,1])}. \]
Combining the above result with \eqref{eq: suboptimal bound 1}, and letting $C_3 = C_1^2C_2$, we see, 
\begin{equation}\label{eq: bound 1}
\|f-s\|^2_{L_2([0,1])}\leq C_3 h^{\alpha}\|{}_0 D_x^{\alpha} f\|_{L^2([0,1])}\|f-s\|_{L^2([0,1])}
\end{equation}
or, 
\begin{equation}\label{eq: bound 2}
\|f-s\|_{L_2([0,1])}\leq C_3 h^{\alpha}\|{}_0 D_x^{\alpha} f\|_{L^2([0,1])}
\end{equation}
as desired.
\end{proof}

\begin{remark}
   A similar result to Theorem \ref{thm: convergence rates} can be derived for the right Riemann-Liouville Green's kernel interpolants. In that case we require $f\in \left\{C^1_0[0,1]: f' \in C_0((0,1])\right\}$ and the error bound obtained will be in terms of $\|{}_x D_1^{\alpha} f\|_{L^2([0,1])}$. The proof, however, is nearly identical.
\end{remark}
\begin{remark}
  If $\alpha = 2$, then $G(x,z)$ is the Brownian bridge kernel and the convergence result is a special case of Theorem 9.1 in \cite{Fass}.
\end{remark}

\section{Kernel Galerkin Methods}\label{sec:galerkin}
In what follows, we assume that we have a set of pairwise distinct sampling centers $\mathcal{X} = \{x_1,x_2,\ldots,x_N\} \subseteq \Omega = (0,1)$.
Let $V = H^{\alpha/2}_0(\overline\Omega)$ and define
\[
V_N =\text{span}\left\{G(\cdot,x_1), G(\cdot,x_2), \ldots, G(\cdot,x_N) \right\}
\]
with $G$ the reproducing kernel of $V$ and Green's kernel of (for instance)
\begin{equation}\label{eq: bvp}
\left\{
\begin{array}{rcl}
     -{}_0 D_x^{\alpha} u(x) & = & f(x), \quad \text{in } (0,1) \\
      u(0) & = & u(1) = 0 \\
\end{array} 
\right. .
\end{equation}
\subsection{Two-point boundary value problems}\label{sec:frac_bvp}
We introduce the Kernel-based Galerkin approximation problem for \eqref{eq: bvp}:

Find $u_N \in V_N$ such that
\[
\int_{\Omega} -{}_0 D_x^{\alpha} u_N  v dx = \int_{\Omega} f v dx, \quad v \in V_N,
\]
where
\[
u_N(x) = \sum_{j=1}^{N} c_j G(x,x_j).
\]
Next, we take $v(\cdot) = G(\cdot,x_i),\ x_i \in \mathcal{X}$, which gives
\[
\begin{aligned}
\sum_{j=1}^N c_j G(x_i,x_j) &= \sum_{j=1}^N c_j \int_{\Omega} \delta_{x_j}(x) G(x,x_i) dx = \int_{\Omega} -{}_0 D_x^{\alpha}u_N v dx \\
&= \int_{\Omega} f(x) G(x,x_i) dx.
\end{aligned}
\]
This leads to a linear system
\[
\text{G} \bold c = \bold f,
\]
where 
\[\text{G}_{ij} = G(x_i,x_j), \, i,j = 1,\ldots,N, \quad \bold c = [c_1,c_2,\ldots,c_N]^T \]
and
\[\bold f = [f_1,f_2,\ldots,f_N]^T, \quad f_i = \int_{\Omega} f(x) G(x,x_i) dx\]
with the integral determined by numerical integration.

\subsection{A fractional diffusion problem}\label{sec:frac_diff}

Let $\{t_n\}_{n=0}^{N_t}$ be a partition of $[T_0,T_1]$, where $t_n  = T_0 + n\tau$ and the time step 
\[ \tau = \frac{T_1 - T_0}{N_t},\] then the kernel-based Galerkin Crank-Nicolson scheme for the fractional diffusion problem \eqref{eq: frac diff 1}-\eqref{eq: frac diff 3} is given as follows. For $n = 1, \dots, N_T-1$, find $u^n_N \in V_N$ such that
\[
\int_{\Omega} d_t u^n_N  v dx  - \int_{\Omega} {}_0 D_x^{\alpha} u^{n+1/2}_N  v dx = \int_{\Omega} f^{n+1/2} v dx, \quad v \in V_N,
\]
where
\[
u^n_N(x) = \sum_{j=1}^{N} c^n_j G(x,x_j), \quad d_t u^n_N = \frac{u^{n+1}_N - u^{n}_N}{\tau}, \text{ and }  \phi^{n+1/2} = \frac{\phi^{n+1} + \phi^{n}}{2}.
\]
Moreover, in matrix-vector form, we can write
\[\left(\text{A}+\frac{\tau}{2}\text{G} \right) \bold c^{n+1} = \left(\text{A} -\frac{\tau}{2}\text{G} \right)\bold c^{n}  + \tau \bold f^{n+1/2},  \]
where 
\[\text{A}_{ij} = \int_{\Omega}G(\cdot,x_i)G(\cdot,x_j)dx \text{ and } \text{G}_{ij} = G(x_i,x_j), \, i,j = 1,\ldots,N, \quad \bold c^n = [c^n_1,c^n_2,\ldots,c^n_N]^T \]
and
\[\bold f^{n+1/2} = [f^{n+1/2}_1,f^{n+1/2}_2,\ldots,f^{n+1/2}_N]^T, \quad f^{n+1/2}_i = \int_{\Omega} f(x, \tau(n + 0.5)) G(x,x_i) dx\]
with all integrals above determined by numerical integration.

\section{Numerical Results}\label{sec: numerics}

\subsection{Interpolation problems}\label{sec: interp probs}
We will solve some basic interpolation problems to test the convergence rate of the left Riemann-Liouville Green's kernel interpolant for functions satisfying different smoothness conditions at the boundaries using
$\alpha \in\{1,1.25,1.5,1.75,2\}$. For a first set of experiments, we uniformly sample the functions 
\begin{equation}\label{eq: test func 1}
f_1(x) = \sin\left(\frac{\pi}{2} x \right)x(1-x)
\end{equation}
and
\begin{equation}\label{eq: test func 2}
f_2(x) = \cos\left(\frac{\pi}{2} x \right)x(1-x)
\end{equation}
at $N = 20,40,80,160,320$ points in $(0,1)$ and evaluate each of the interpolants at 
$N_{eval} = 1000$ points in $[0,1]$. Note that the functions \eqref{eq: test func 1} 
and \eqref{eq: test func 2} satisfy the following boundary conditions,
\[f_1(0) = f_1(1) = f_2(0) = f_2(1) = 0\]
with 
\[f'_1(0) = f'_2(1) = 0,\]
and
\[-f'_1(1) = f'_2(0) = 1.\]
Thus, both $f_1$ and $f_2$ vanish at the boundary, but choosing between the two functions allows us to test how the boundary conditions of their derivatives affect convergence rates of their corresponding interpolants. To that end, for each value of $\alpha$, and each value of $N$, we measure the root mean squared error
\begin{equation}\label{eq:rmse}
Error = \left(\frac{1}{N_{eval}}\sum_{i=1}^{N_{eval}}(s(x_i) - f(x_i))^2\right)^{1/2}.\end{equation}
Given that $f_1'$ vanishes at the left boundary but $f_2'$ does not, we only expect to obtain optimal convergence rates for the left Riemann-Liouville Green's kernel interpolant of $f_1$. As predicted, the convergence rate for the interpolant of $f_1$ is computed to be $\mathcal{O}(h^{\alpha})$, but for the interpolant of $f_2$ this optimal convergence rate is not guaranteed. These results are shown in Figure \ref{fig:rates}.

\begin{figure}[ht]
\centering
\includegraphics[width=0.45\textwidth]{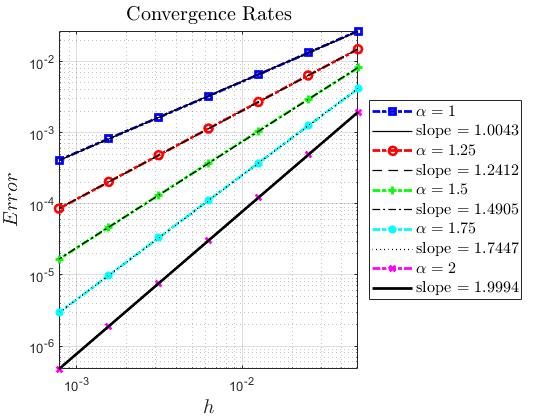}
\includegraphics[width=0.45\textwidth]{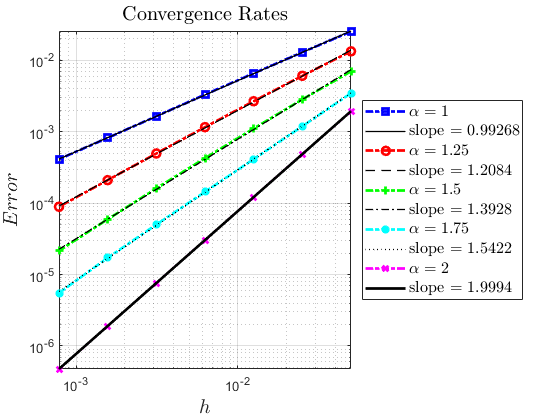}
\caption{Errors and convergence rates for left fractional Riemann-Liouville Green's Green's kernel interpolation sampled at uniform points for $f_1$ (left) and $f_2$ (right).}
\label{fig:rates}
\end{figure}

For the next set of experiments, we once again sample the functions $f_1$ in \eqref{eq: test func 1} and $f_2$ in \eqref{eq: test func 2}, but the Chebyshev nodes are used instead of uniform sampling centers. All of the above computations are repeated and the results are shown in Figure \ref{fig:rates2}. These numerical results indicate that the choice of Chebyshev nodes can lead to optimal convergence rates even when the data are sampled from a function which does not satisfy the hypothese of Theorem  \ref{thm: convergence rates}.

\begin{figure}[ht]
\centering
\includegraphics[width=0.45\textwidth]{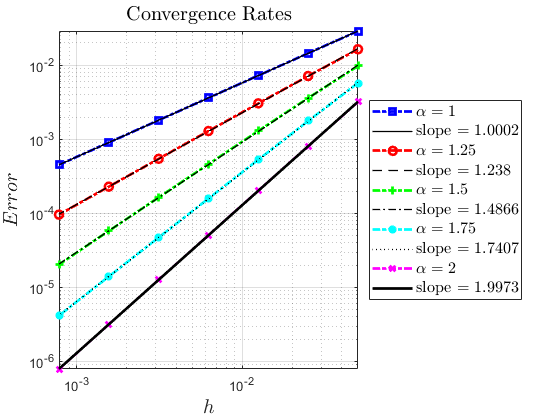}
\includegraphics[width=0.45\textwidth]{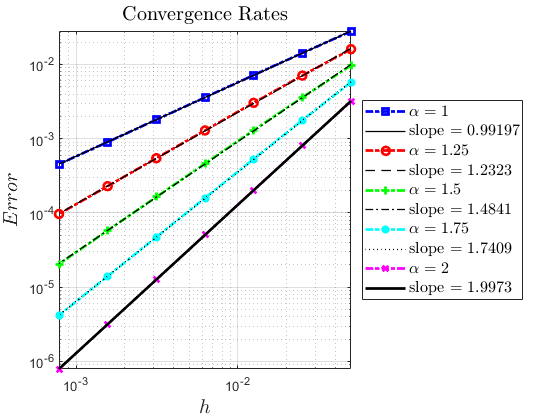}
\caption{Errors and convergence rates for left fractional Riemann-Liouville Green's Green's kernel interpolation sampled at Chebyshev points for $f_1$ (left) and $f_2$ (right).}
\label{fig:rates2}
\end{figure}

\subsection{Boundary value problems} \label{sec:example bvps}
In this section we consider two fractional boundary value problems corresponding to left and right fractional Riemann-Liouville derivatives, respectively.
\paragraph{Problem 1: Left fractional derivatives} 
\text{  }\\

Note that the left fractional Riemann-Liouville derivative of the monomial $(x-a)^\gamma$ is given by 
\[{}_a D_x^\alpha(x-a)^{\gamma} = \frac{\Gamma(\gamma +1)(x-a)^{\gamma - \alpha}}{\Gamma(\gamma +1 - \alpha)}.\]
Thus, if we let 
\[f(x) =-\frac{\Gamma(2)x^{1 - \alpha}}{\Gamma(2 - \alpha)}+ \frac{\Gamma(3)x^{2- \alpha}}{\Gamma(3 - \alpha)} \]
then, 
\[u(x) =x-x^2 \]
is the exact solution to the boundary value problem, 
\[
\left\{
\begin{array}{rcl}
     -{}_0 D_x^{\alpha} u(x) & = & f(x), \quad \text{in } (0,1) \\
      u(0) & = & u(1) = 0 \\
\end{array} 
\right. .
\]

\paragraph{Problem 2: Right fractional derivatives} 
\text{  }\\

Consider the right fractional Riemann-Liouville derivatives
\[{}_x D_1^\alpha x = \frac{\Gamma(\alpha+1)(1-x)^{- \alpha}}{\Gamma(2 - \alpha)}\left( \frac{x}{\Gamma(\alpha+1)} - \frac{1}{\Gamma(\alpha)}\right),\]
and 
\[{}_x D_1^\alpha x^2 = \frac{2\Gamma(\alpha+1)(1-x)^{- \alpha}}{\Gamma(3 - \alpha)}\left( \frac{x^2}{\Gamma(\alpha+1)} -\frac{x}{\Gamma(\alpha)} + \frac{1}{2\Gamma(\alpha-1)}\right).\]
Thus, if we let 
\begin{multline*}
f(x) =\frac{\Gamma(\alpha+1)(1-x)^{- \alpha}}{\Gamma(2 - \alpha)}\left( \frac{x}{\Gamma(\alpha+1)} - \frac{1}{\Gamma(\alpha)}\right)- \\ 
\frac{2\Gamma(\alpha+1)(1-x)^{- \alpha}}{\Gamma(3 - \alpha)}\left( \frac{x^2}{\Gamma(\alpha+1)} -\frac{x}{\Gamma(\alpha)} + \frac{1}{2\Gamma(\alpha-1)}\right), 
\end{multline*}
then, 
\[u(x) =x-x^2 \]
is the exact solution to the boundary value problem, 
\[
\left\{
\begin{array}{rcl}
     -{}_x D_1^{\alpha} u(x) & = & f(x), \quad \text{in } (0,1) \\
      u(0) & = & u(1) = 0 \\
\end{array} 
\right. .
\]

Using $N = 20,40,80,160,320$ points in $(0,1)$ and for $\alpha \in \{ 1.5, 1.75\}$, we test the Galerkin method of Section \ref{sec:frac_bvp} using the above test problems. For both problems, we sample using both uniform and Chebyshev nodes and we compute the root mean squared error \eqref{eq:rmse} by evaluating the solution at $N_{eval} = 1000$ uniform points. Convergence rates are computed using the formula
\begin{equation}\label{eq:rate}
Rate = \frac{\log(\epsilon_{h_1}/\epsilon_{h_2})}{\log(h_2/h_1)} ,
\end{equation}
where $h_1> h_2$ are two consecutive values of the fill distance $h$. In the case of uniform sampling, the error is computed on both the entire domain, $[0,1]$, and on the interior of the domain, $[0.01,0.99]$. For the Chebyshev nodes, the error is computed on the entire domain. The error and convergence rates are reported in Table \ref{tab:galerkin rates} and Table \ref{tab:galerkin rates 2} for Problem 1 and Problem 2, respectively. For Problem 1, the numerical evidence suggests that the choice of Chebyshev nodes guarantees optimal convergence rates on the entire domain, while uniform sampling centers are potentially subject to strong boundary effects. For Problem 2, the choice of sampling centers is less relevant.

\begin{table}
    \centering
    \begin{tabular}{|c|c|c||c|c|c||c|c|c|}
        \hline
        \multicolumn{9}{|c|}{$\alpha = 1.5$} \\
        \hline
        \multicolumn{3}{|c||}{Uniform} & \multicolumn{3}{|c||}{Uniform (interior) } & \multicolumn{3}{|c|}{Chebyshev}\\
        \hline
        N   & Error      & Rate       & N   & Error      & Rate       & N   & Error      & Rate \\
        \hline \hline
        20  & 3.7361e-03 & --         & 20 & 3.6477e-03  & --         & 20  & 4.1877e-03 & --     \\
         \hline
        40  & 1.3654e-03 & 1.3998     & 40 & 1.2638e-03  & 1.4740     & 40  & 1.4428e-03 & 1.4835 \\
         \hline
        80  & 5.0519e-04 & 1.4085     & 80 & 4.4054e-04  & 1.4929     & 80  & 5.0261e-04 & 1.4943 \\
         \hline
        160 & 1.8504e-04 & 1.4359     & 160 & 1.5528e-04 & 1.4909     & 160 & 1.7664e-04 & 1.4952 \\
         \hline
        320 & 6.7506e-05 & 1.4482     & 320 & 5.5581e-05 & 1.4755     & 320 & 6.2167e-05 & 1.4998 \\
        \hline
        \hline
        \multicolumn{9}{|c|}{$\alpha = 1.75$} \\
        \hline
        \multicolumn{3}{|c||}{Uniform} & \multicolumn{3}{|c||}{Uniform (interior) } & \multicolumn{3}{|c|}{Chebyshev}\\
        \hline
        N   & Error      & Rate   & N   & Error      & Rate   & N   & Error      & Rate \\
        \hline \hline
        20  & 1.4599e-03 & --     & 20 & 1.4359e-03  & --     & 20  & 2.0305e-03 & --     \\
         \hline
        40  & 4.7161e-04 & 1.5713 & 40 & 4.2605e-04  & 1.6896 & 40  & 5.8042e-04 & 1.7435 \\
         \hline
        80  & 1.5685e-04 & 1.5595 & 80 & 1.1786e-04  & 1.8205 & 80  & 1.6921e-04 & 1.7467 \\
         \hline
        160 & 5.3117e-05 & 1.5481 & 160 & 3.5078e-05 & 1.7326 & 160 & 4.9805e-05 & 1.7486 \\
         \hline
        320 & 1.7888e-05 & 1.5631 & 320 & 1.0396e-05 & 1.7466 & 320 & 1.4725e-05 & 1.7501 \\
         \hline
    \end{tabular}
  \caption{Errors and convergence rates of the kernel Galerkin method for Problem 1 (the left Riemann-Liouvile Boundary value problem) using $\alpha = 1.5$ and $\alpha = 1.75$. For each value of $\alpha$, both uniform and Chebyshev sampling nodes are used. In the case of uniform sampling, the error is computed on both the entire domain (left column) and on the interior of the domain (center column). For the Chebyshev nodes, the error is computed on the entire domain (right column).}
    \label{tab:galerkin rates}
\end{table}

\begin{table}
    \centering
    \begin{tabular}{|c|c|c||c|c|c||c|c|c|}
        \hline
        \multicolumn{9}{|c|}{$\alpha = 1.5$} \\
        \hline
        \multicolumn{3}{|c||}{Uniform} & \multicolumn{3}{|c||}{Uniform (interior) } & \multicolumn{3}{|c|}{Chebyshev}\\
        \hline
        N   & Error      & Rate       & N   & Error      & Rate       & N   & Error      & Rate \\
        \hline \hline
        20  & 1.1691e-02 & --         & 20 & 1.1803e-02  & --         & 20  & 1.1622e-02 & --     \\
         \hline
        40  & 5.0108e-03 & 1.1782     & 40 & 5.0242e-03  & 1.1877   & 40  & 4.5013e-03 & 1.3207\\
         \hline
        80  & 1.7283e-0  & 1.5080     & 80 & 1.7159e-03  & 1.5220     & 80  & 1.6973e-03 & 1.3821 \\
         \hline
        160 & 6.4719e-04 & 1.4043     & 160 & 6.5063e-04 & 1.3864     & 160 & 6.0997e-04 & 1.4632 \\
         \hline
        320 & 2.3058e-04 & 1.4822     & 320 & 2.2876e-04 & 1.5012     & 320 & 2.1954e-04 & 1.4676  \\
        \hline
        \hline
        \multicolumn{9}{|c|}{$\alpha = 1.75$} \\
        \hline
        \multicolumn{3}{|c||}{Uniform} & \multicolumn{3}{|c||}{Uniform (interior) } & \multicolumn{3}{|c|}{Chebyshev}\\
        \hline
        N   & Error      & Rate   & N   & Error      & Rate   & N   & Error      & Rate \\
        \hline \hline
        20  & 8.5524e-03 & --     & 20 & 8.6356e-03   & --     & 20  & 7.6402e-03 & --     \\
         \hline
        40  & 2.9699e-03 & 1.4708 & 40 & 2.9954e-03  & 1.4724  & 40  & 2.5964e-03 & 1.5027  \\
         \hline
        80  & 8.7020e-04 & 1.7390 & 80 & 8.7004e-04  & 1.7514  & 80  & 8.2594e-04 & 1.6230 \\
         \hline
        160 & 2.7861e-04 & 1.6283 & 160 & 2.7914e-04 & 1.6253  & 160 & 2.5246e-04 & 1.6947 \\
         \hline
        320 & 8.2125e-05 & 1.7544 & 320 & 8.1570e-05 & 1.7669  & 320 & 7.6340e-05 & 1.7178 \\
         \hline
    \end{tabular}
  \caption{Errors and convergence rates of the kernel Galerkin method for Problem 2 (the right Riemann-Liouvile Boundary value problem) using $\alpha = 1.5$ and $\alpha = 1.75$. For each value of $\alpha$, both uniform and Chebyshev sampling nodes are used. In the case of uniform sampling, the error is computed on both the entire domain (left column) and on the interior of the domain (center column). For the Chebyshev nodes, the error is computed on the entire domain (right column).}
    \label{tab:galerkin rates 2}
\end{table}

\subsection{Fractional diffusion equation in one space variable}
As a final numerical example, we consider a modified version of the benchmark problem that appears in \cite{PiretHanert}:
\begin{equation}\label{eq:diffusion}
    \frac{\partial u(x,t)}{\partial t} = d(x){}_0 D_x^{\alpha}u(x,t) + q(x,t) \quad \text{ for } x \in [0,1] \text{ and } t>0
\end{equation}
with $d(x) = 1$, 
$q(x,t) = -e^{-t}\left( \left(1 + \Gamma(2)\frac{x^{-\alpha}}{\Gamma(2-\alpha)}\right)x  -  \left(1  +  \Gamma(5)\frac{x^{-\alpha}}{\Gamma(5-\alpha)}\right)x^4 \right)$, $u(x,0) = x-x^4$, and $u(0,t) = u(1,t) = 0$ so that the exact solution to  \eqref{eq:diffusion} is
\[u(x,t) = e^{-t}(x-x^4).\]

Using $N = 20,40,80,160,320$ points in $(0,1)$ and $\alpha = 1.3, \, 1.8$, we test the Galerkin method of Section \ref{sec:frac_diff} using the test problem \eqref{eq:diffusion}. For each value of $N$ we set $[T_0, T_1] = [0,1]$ and select $N_t = N$. Moreover, for these values of $N$ and $N_T$ we expect the overall error in the numerical scheme to be $\mathcal{O}(h^\alpha)$. As in Problem 1 of Section \ref{sec:example bvps}, we sample using both uniform and Chebyshev nodes and we compute the root mean squared error \eqref{eq:rmse} at the last time step by evaluating the solution at $N_{eval} = 1000$ uniform points and calculate convergence rates using \eqref{eq:rate}. In the case of uniform sampling, the error in the final time step is computed on both the entire domain, $[0,1]$, and on the interior of the domain, $[0.01,0.99]$. For the Chebyshev nodes, the error is computed on the entire domain. The error and convergence rates are reported in Table \ref{tab:galerkin rates 3}. Once more, the numerical evidence suggests that the choice of Chebyshev nodes guarantees optimal convergence rates on the entire domain, while uniform sampling centers are potentially subject to strong boundary effects.  

\begin{table}
    \centering
    \begin{tabular}{|c|c|c||c|c|c||c|c|c|}
        \hline
        \multicolumn{9}{|c|}{$\alpha = 1.3$} \\
        \hline
        \multicolumn{3}{|c||}{Uniform} & \multicolumn{3}{|c||}{Uniform (interior) } & \multicolumn{3}{|c|}{Chebyshev}\\
        \hline
        N   & Error      & Rate       & N   & Error      & Rate       & N   & Error      & Rate \\
        \hline \hline
        20  & 5.2718e-03 & --         & 20 & 4.9678e-03  & --         & 20  & 5.6653e-03 & --     \\
         \hline
        40  & 2.1115e-03 & 1.2723 & 40 & 1.8827e-03  & 1.3492 & 40  & 2.2488e-03 & 1.2864 \\
         \hline
        80  & 8.9546e-04 & 1.2153 & 80 & 7.8723e-04  & 1.2353 & 80  & 8.8911e-04 & 1.3149 \\
         \hline
        160 & 3.6902e-04 & 1.2674 & 160 & 3.1850e-04 & 1.2937 & 160 & 3.5418e-04 & 1.3160 \\
         \hline
        320 & 1.4868e-04 & 1.3056 & 320 & 1.2965e-04 & 1.2908 & 320 & 1.4004e-04 & 1.3326 \\
        \hline
        \hline
        \multicolumn{9}{|c|}{$\alpha = 1.8$} \\
        \hline
        \multicolumn{3}{|c||}{Uniform} & \multicolumn{3}{|c||}{Uniform (interior) } & \multicolumn{3}{|c|}{Chebyshev}\\
        \hline
        N   & Error      & Rate   & N   & Error      & Rate   & N   & Error      & Rate \\
        \hline \hline
        20  & 9.4843e-04 & --     & 20 & 9.3055e-04  & --     & 20  & 1.4861e-03 & --     \\
         \hline
        40  & 3.1450e-04 & 1.5350 & 40 & 2.7441e-04  & 1.6981 & 40  & 4.0677e-04 & 1.8039 \\
         \hline
        80  & 1.0880e-04 & 1.5038 & 80 & 7.0093e-05  & 1.9335 & 80  & 1.1302e-04 & 1.8148 \\
         \hline
        160 & 3.8074e-05 & 1.5011 & 160 & 1.9642e-05 & 1.8188 & 160 & 3.1477e-05 & 1.8277 \\
         \hline
        320 & 1.3106e-05 & 1.5316 & 320 & 5.3384e-06 & 1.8710 & 320 & 8.6916e-06 & 1.8483 \\
         \hline
    \end{tabular}
    \caption{Errors and convergence rates of the kernel Galerkin method for the fractional diffusion problem using $\alpha = 1.3$ and $\alpha = 1.8$. For each value of $\alpha$, both uniform and Chebyshev sampling nodes are used. In the case of uniform sampling, the error in the final time step is computed on both the entire domain (left column) and on the interior of the domain (center column). For the Chebyshev nodes, the error is computed on the entire domain (right column).}
    \label{tab:galerkin rates 3}
\end{table}

\section{Conclusions and Future Work}

In short, the above results indicate that Riemann-Liouville Green's kernels offer a simple yet flexible means for fractional order interpolation in reproducing kernel Banach spaces. Indeed, optimal order convergence rates for the kernel interpolation problem were derived theoretically and demonstrated numerically. Moreover, the corresponding kernel Galerkin method was applied to the numerical solution of a fractional diffusion problem. Not only does the kernel Galerkin method obtain the expected convergence rates, but it does so with requiring the computation or approximation of any fractional derivatives.

Thus far, we have only considered the Green's kernels corresponding to boundary value problems of the form
\[
\left\{
\begin{array}{rcl}
     -{}_0 D_x^{\alpha} u(x) & = & f(x,u(x)), \quad \text{ in } \Omega \\
       u(0) & = & u(1) = 0 \\
\end{array} 
\right.
\]
but it is well known that, in the integer order case, an entire family of iterated Brownian bridge kernels can be generated \cite{Cavoretto}. Evidently, this would correspond to fractional operators of the following type
\[\mathcal{L} = \left(-{}_0 D_x^{\alpha} + \varepsilon^2 \mathcal{I} \right)^\beta\]
for some appropriate boundary conditions. However, to the author's knowledge, the corresponding Green's functions are still unknown \cite{Jin}. Hence, this is left for future work. Obvious extensions to higher dimensions via tensor products could also be considered.

\section*{Acknowledgements}
This research was supported by the NSF RTG grant DMS-2136228. The author would like to thank Greg Fasshauer and Yulong Li for their helpful discussions in the preparation of this manuscript as well as Michael Neunteufel and Julie Zhu for their careful readings of an early draft.

\bibliographystyle{siam} 
\bibliography{fractional_rkbs}

\end{document}